\documentclass[11pt]{article}
\usepackage{etex}
\usepackage{pgf,tikz}
\usetikzlibrary{arrows}
\usepackage[margin=1in]{geometry}
\usepackage{stackengine}

\usepackage[utf8]{inputenc}
\usepackage{color}
\usepackage[T1]{fontenc}
\usepackage[normalem]{ulem}
\usepackage[english]{babel}
\usepackage{verbatim}
\usepackage{graphicx}
\usepackage{enumerate}
\usepackage{amsmath,amssymb,amsfonts,amsthm,amscd,mathrsfs}
\usepackage{array}
\usepackage{amsmath,amssymb,graphicx,stmaryrd,enumerate,bbm,alltt}
\usepackage{dsfont}
\usepackage{comment}
\usepackage{mathtools}
\usepackage{bbm}

\usepackage[T1]{fontenc}
\usepackage{babel}
\usepackage[bookmarksopen, bookmarksnumbered]{hyperref}
\usepackage[page]{appendix}

\usepackage{url}
\usepackage{caption}

\newtheorem{theorem}{Theorem}[section]

\newtheorem{lemma}[theorem]{Lemma}
\newtheorem{cla}[theorem]{Claim}
\newtheorem{prop}[theorem]{Proposition}

\theoremstyle{definition}
\newtheorem{defn}[theorem]{Definition}
\newtheorem{rem}[theorem]{Remark}

\newcommand{\Z}{\mathbb Z}
\newcommand{\R}{\mathbb R}
\newcommand{\E}{\mathbb E}
\newcommand{\Prob}{\mathbb P}

\title{On the Manhattan pinball problem}
\author{Linjun Li
\thanks{Department of Mathematics, University of Pennsylvania, Philadelphia, PA. e-mail:
 linjun@sas.upenn.edu} 
}
\date{}

\begin{document}

\maketitle
\begin{abstract}
We consider the periodic Manhattan lattice with alternating orientations going north-south and
east-west. Place obstructions on vertices independently with probability $0<p<1$. A particle is moving on the edges with unit speed following the orientation of the lattice and it will turn only when encountering an obstruction. The problem is that for which value of $p$ is the trajectory of the particle closed almost surely. We prove this is true for $p>\frac{1}{2}-\varepsilon$ with some $\varepsilon>0$. 
\end{abstract}

\section{Introduction}
We consider the Manhattan pinball problem which is a network model of a quantum localization problem (see \cite{beamond2003quantum} and also \cite[Page 237]{spencer2012duality},\cite{cardy2010quantum}). There are equivalent ways to state the problem formally. We follow \cite[Page 238]{spencer2012duality} and state it by using bond percolation. Consider the $\Z^{2}$ lattice embedded into the plane $\R^{2}$. Denote the tilted lattice by $$\widetilde{\Z^{2}}=\left\{\left(x+\frac{1}{2},y+\frac{1}{2}\right):\text{$x,y$ are integers and $x-y$ is even}\right\},$$
and for any $a,b\in \widetilde{\Z^{2}}$, $a$ and $b$ are connected by an edge if and only if $|a-b|=\sqrt{2}$. Here, $|\cdot|$ is the Euclidean distance on $\R^{2}$.
See Figure \ref{fig:2} for an illustration.
\begin{figure}
    \centering
    \includegraphics{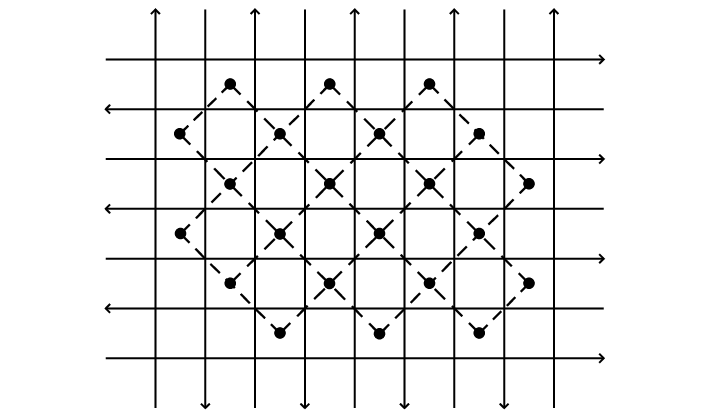}
    \caption{The black dots in the picture illustrate the lattice $\widetilde{\Z^{2}}$. The dashed lines are edges of $\widetilde{\Z^{2}}$. The arrows indicate the orientation of the Manhattan lattice.}
    \label{fig:2}
\end{figure}
Given $0<p<1$, we consider the \emph{Bernoulli bond percolation} on $\widetilde{\Z^{2}}$. We declare each edge of $\widetilde{\Z^{2}}$ to be closed with probability $p$ and open with probability $1-p$, independent of all other edges. We use $\omega$ to denote a configuration of open and closed edges and then place a two-sided plane mirror on each closed edge in $\omega$. Suppose a ray of light starts from the origin and initially moves towards one of the two directions following the Manhattan orientation. When the light reaches an open edge, it passes through without deflection. When the light reaches a closed edge, it is deflected through a right angle by the mirror on the edge. Let $L(\omega)$ denote the trajectory of the light and we view $L(w)$ as a piece-wise linear path in $\mathbb{R}^{2}$.
\begin{defn}\label{def:n_bounded}
Given a nonnegative integer $n$, denote by
$$Q_{n}=\left\{(x,y)\in \widetilde{\Z^{2}}: \text{$|x+y-1|\leq n$ and $|x-y|\leq n$}  \right\}$$
the tilted box centered at $(\frac{1}{2},\frac{1}{2})$. We say a trajectory $L(w)$ is \emph{$n$-bounded} if any point $(x,y)$ on $L(w)$ satisfies $|x+y-1|\leq n$ and $|x-y|\leq n$.
\end{defn}

Our main result is the following
\begin{theorem}\label{thm:main}
Let $\mathcal{E}_{n}$ denote the event that $L(\omega)$ is $n$-bounded. Then there exists $\varepsilon_{0}>0$ such that following holds. For each $p>\frac{1}{2}-\varepsilon_{0}$, there are $\alpha,c>0$ such that for $n>\alpha$, we have
\begin{equation}
    \Prob_{p} [\mathcal{E}_{n}]\geq 1-\exp(-c n).
\end{equation}
\end{theorem}
Theorem \ref{thm:main} implies that almost surely, the trajectory of the light is bounded in $\R^{2}$ when $p>\frac{1}{2}-\varepsilon_{0}$. This result
is previously known to hold for $p>\frac{1}{2}$ by using bond percolation, see e.g. \cite{beamond2003quantum}, \cite[Page 238]{spencer2012duality}. It is suggested by physicists that the same happens for any $p\in (0,1)$ (see \cite{beamond2003quantum}) and the average diameter of the trajectory is conjectured to scale as $\exp(c p^{-2})$ when $p$ tends to $0$. 
Interestingly, this assertion is in agreement with the prediction that all states in the two dimensional \emph{random Schr\"{o}dinger model} are localized (see \cite{abrahams1979scaling}).
For mathematical results on 2D Anderson localization with singular potential, see e.g. \cite{carmona1987anderson},\cite{bourgain2005localization},\cite{ding2020localization},\cite{li2020anderson}. 
We also refer the reader to \cite{spencer2012duality} for a relationship between Manhattan pinball problem and the quantum localization problem. 
\section{Enhancement}
\begin{figure}
    \centering
    \includegraphics{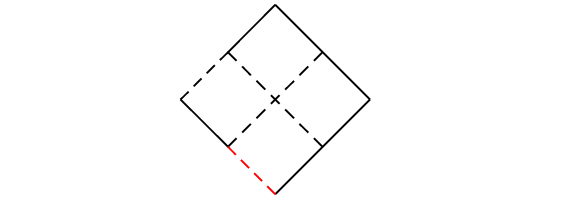}
    \caption{The figure illustrates the configuration $G_{0}$ where solid black edges are closed and dashed edges are open.}
    \label{fig:1}
\end{figure}
The proof of Theorem \ref{thm:main} uses an \emph{enhancement} procedure which we define now.
Denote by $G_{0}$ the finite configuration illustrated in Figure \ref{fig:1} and its caption.
Given any configuration $\omega$, we enhance it by the following procedure. For any translated  copy of $G_{0}$ which appears in $\omega$, we close the open edge which coincides with the red edge (illustrated in Figure \ref{fig:1}). We denote the configuration after this procedure by $\tilde{\omega}$.

The procedure above is a special case of the general enhancement (see e.g. \cite[Section 3.3]{grimmett1999percolation} and \cite{aizenman1991strict}). Intuitively, we have added closed edges to $\omega$ in a systematic way and we expect that the closed edges of resulting configuration $\tilde{\omega}$ are more ``percolative'' than closed edges of $\omega$.
This is an informal statement of \cite[Theorem (3.16)]{grimmett1999percolation}. For \cite[Theorem (3.16)]{grimmett1999percolation} to hold, the enhancement is required to be \emph{essential} (a concept defined in \cite[Page 64]{grimmett1999percolation}) which is true for our case (see Figure \ref{fig:5} and its caption).

\begin{figure}
    \centering
    \includegraphics{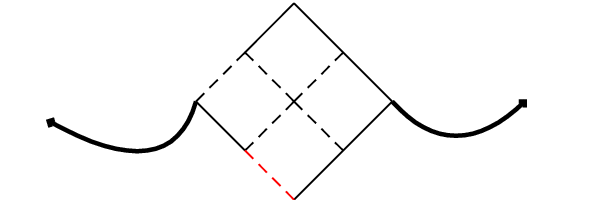}
    \caption{The picture contains a doubly-infinite closed path if and only if the enhancement is activated so that the red edge is turned to be closed. This means the enhancement is essential.}
    \label{fig:5}
\end{figure}
\begin{rem}
Note that our enhancement is not \emph{monotonic} in the sense of \cite[Section 3.3]{grimmett1999percolation}. However, \cite[Theorem (3.16)]{grimmett1999percolation} does not require monotonicity.
\end{rem}
\begin{rem}
There are other choices of the enhancement $G_{0}$ for our proof of Theorem \ref{thm:main}. We will see in Section \ref{sec:proof} that, our choice of $G_{0}$ makes the proof easier because of an important feature contained in Lemma \ref{lem:translation} below.
\end{rem}
\begin{lemma}\label{lem:translation}
Suppose $\omega$ is a configuration and $\overline{G_{0}}$, $\overline{\overline{G_{0}}}$ are two different translated copies of $G_{0}$ that appear in $\omega$. Then any open edge in $\overline{G_{0}}$ does not coincide with the red edge in $\overline{\overline{G_{0}}}$.
\end{lemma}
\begin{proof}
For simplicity,
we first label six edges $e_{i}$ ($0\leq i\leq 5$) in $G_{0}$ as illustrated in Figure \ref{fig:labeling-edges} and its caption. 
\begin{figure}
    \centering
    \includegraphics{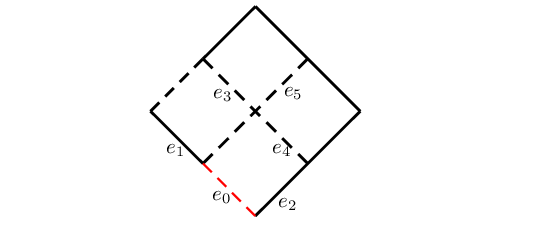}
    \caption{Illustration of the labelling of six edges $e_{i}$ ($0\leq i\leq 5$) in $G_{0}$ where $e_{0}$ is the red edge. $e_{1}$ and $e_{2}$ are closed while $e_{3}$, $e_{4}$, $e_{5}$ and $e_{0}$ are open.}
    \label{fig:labeling-edges}
\end{figure}
For $0\leq i\leq 5$,
denote by $\overline{e_{i}}$ the translation of $e_{i}$ in $\overline{ G_{0}}$ and denote by $\overline{\overline{e_{i}}}$ the translation of $e_{i}$ in $\overline{\overline{ G_{0}}}$.

Assume the red edge in $\overline{\overline{G_{0}}}$ (i.e. $\overline{\overline{e_{0}}}$) coincides with an open edge in $\overline{G_{0}}$ and we will arrive at contradiction. Firstly, by the direction of edge $\overline{\overline{e_{0}}}$ and since $\overline{\overline{e_{0}}}$ is open, we must have $\overline{\overline{e_{0}}}=\overline{e_{0}}$ or
$\overline{\overline{e_{0}}}=\overline{e_{3}}$ or $\overline{\overline{e_{0}}}=\overline{e_{4}}$. Since $\overline{G_{0}}\neq \overline{\overline{G_{0}}}$, we have $\overline{\overline{e_{0}}}\neq\overline{e_{0}}$.
If $\overline{\overline{e_{0}}}=\overline{e_{3}}$, then we have $\overline{\overline{e_{2}}}=\overline{e_{5}}$ which contradicts with the fact that $\overline{\overline{e_{2}}}$ is closed and $\overline{e_{5}}$ is open. If $\overline{\overline{e_{0}}}=\overline{e_{4}}$, then we have $\overline{\overline{e_{1}}}=\overline{e_{3}}$ which contradicts with the fact that $\overline{\overline{e_{1}}}$ is closed and $\overline{e_{3}}$ is open. Thus we arrive at contradiction and our lemma follows.
\end{proof}
\begin{defn}
For each positive integer $n$,
let $A_{n}$ denote the event that there exists a path of closed edges joining $(\frac{1}{2},\frac{1}{2})$ to some vertex in $\partial Q_{n}$ where $\partial Q_{n}=Q_{n}\setminus Q_{n-1}$.
\end{defn}
The following proposition follows from the proof of \cite[Theorem (3.16)]{grimmett1999percolation} and the fact that the enhancement which we constructed is essential. For more details on adapting arguments in the proof of \cite[Theorem (3.16)]{grimmett1999percolation} to prove Proposition \ref{prop:event-A_n}, see Appendix \ref{app:enh}.
\begin{prop}[Theorem (3.16) in \cite{grimmett1999percolation}]\label{prop:event-A_n}
There exists $\varepsilon_{1}>0$ such that, for each $p>\frac{1}{2}-\varepsilon_{1}$,
\begin{equation}
    \Prob_{p}[\{\omega : \tilde{\omega}\in A_{n}\}] \geq \Prob_{\frac{1}{2}+\varepsilon_{1}}[A_{n}]
\end{equation}
for large enough $n$.
\end{prop}
\begin{defn}
For any positive integer $n$,
let $A'_{n}$ denote the event that there is a path of closed edges in the tilted rectangle
$$T_{n}=\{(x,y)\in \widetilde{\Z^{2}}:1\leq x+y-1\leq n, |x-y| \leq 2n \}$$ joining some vertex on northwest side $\left\{(x,y)\in T_{n}: x-y= -2n\right\}$ to some vertex on southeast side $\left\{(x,y)\in T_{n}: x-y= 2n \right\}$. See Figure \ref{fig:3} for an illustration.
\end{defn}
In fact, by the same argument in Appendix \ref{app:enh} (see Remark \ref{rem:A'_n_case}), we can substitute $A_{n}$ in Proposition \ref{prop:event-A_n} by $A'_{n}$ and we have
\begin{prop}\label{prop:event-A'_n}
There exists $\varepsilon_{1}>0$ such that, for each $p>\frac{1}{2}-\varepsilon_{1}$,
\begin{equation}
    \Prob_{p}[\{\omega : \tilde{\omega}\in A'_{n}\}] \geq \Prob_{\frac{1}{2}+\varepsilon_{1}}[A'_{n}]
\end{equation}
for large enough $n$. 
\end{prop}
The following well-known lemma states that in the supercritical phase (i.e. $p>\frac{1}{2}$), the crossing event $A'_{n}$ happens with high probability. 
\begin{lemma}\label{lem:prob-crossing}
For any $p>\frac{1}{2}$, there is constant $c>0$ such that $$\Prob_{p}[A'_{n}]\geq 1-\exp(-c n)$$ for large enough $n$.
\end{lemma}
\begin{proof}
Follow the argument prior to \cite[Lemma (11.22)]{grimmett1999percolation}.
\end{proof}
Let $A''_{n}$ denote the event that $Q_{n}$ lies in the interior of a circuit which consists of closed edges in $Q_{2n}$.
Together with Proposition \ref{prop:event-A'_n} and Lemma \ref{lem:prob-crossing}, we have
\begin{figure}
    \centering
    \includegraphics{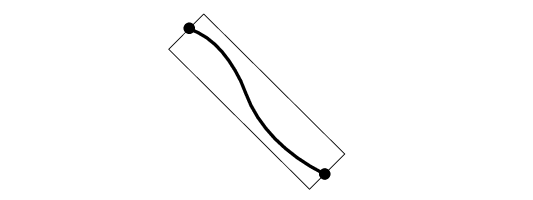}
    \caption{The picture illustrates a closed path in $T_{n}$ which joins some vertex on northwest side to southeast side.}
    \label{fig:3}
\end{figure}
\begin{prop}\label{prop:circuit}
There exists $\varepsilon_{1},c_{1}>0$ such that, for each $p>\frac{1}{2}-\varepsilon_{1}$,
\begin{equation}
    \Prob_{p}[\{\omega : \tilde{\omega}\in A''_{n}\}] \geq 1-\exp(-c_{1} n)
\end{equation}
for large enough $n$. 
\end{prop}
\begin{proof}
Consider the following four tilted rectangles contained in $Q_{2n}$:
\begin{enumerate}
    \item $T^{(1)}=\{(x,y)\in \widetilde{\Z^{2}}: n+1\leq x+y-1\leq 2n, |x-y|\leq 2n\}$,
    \item $T^{(2)}=\{(x,y)\in \widetilde{\Z^{2}}: -2n\leq x+y-1\leq -n-1, |x-y|\leq 2n\}$,
    \item  $T^{(3)}=\{(x,y)\in \widetilde{\Z^{2}}:n+1 \leq x-y\leq 2n,  |x+y-1|\leq 2n\}$,
    \item  $T^{(4)}=\{(x,y)\in \widetilde{\Z^{2}}:-2n\leq x-y\leq -n-1, |x+y-1|\leq 2n\}$. 
\end{enumerate}
See Figure \ref{fig:4} for an illustration. If these four tilted rectangles are crossed in the `long direction' by closed paths as indicated in Figure \ref{fig:4}, then $Q_{n}$ lies in the interior of a closed circuit. By symmetry, each of these four crossing events has the same probability as $A'_{n}$.
Thus our proposition follows from Proposition \ref{prop:event-A'_n} and Lemma \ref{lem:prob-crossing}.
\end{proof}
\begin{figure}
    \centering
    \includegraphics{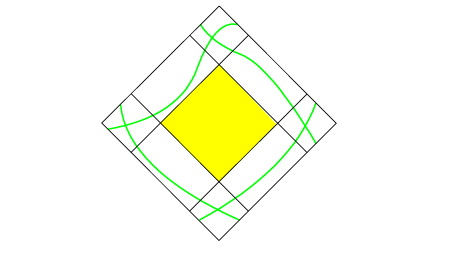}
    \caption{The four green lines are four closed paths that cross tilted rectangles $T^{(i)}(i=1,2,3,4)$ in $Q_{2n}$. If such paths exist, then $Q_{n}$ (the yellow region) lies in the interior of a closed circuit.}
    \label{fig:4}
\end{figure}
\section{Proof of Theorem \ref{thm:main}}\label{sec:proof}
\begin{defn}
Denote by $\widetilde{\E^{2}}$ the set of edges of $\widetilde{\Z^{2}}$. Given an edge $e\in \widetilde{\E^{2}}$ connecting $x,y \in \widetilde{\Z^{2}}$, we also think of $e$ as the segment in $\R^{2}$ whose end points are $x,y$. Given a configuration $\omega$ and $e\in \widetilde{\E^{2}}$,
we say $L(\omega)$ \emph{intersects} $e$ (or equivalently, $e$ intersects $L(\omega)$) if $L(\omega)$ (as a path) intersects $e$ (as a segment).
For any $e\in \widetilde{\E^{2}}$, we say $e$ is \emph{inside} $Q_{n}$ if $e$ connects two vertices in $Q_{n}$, otherwise we say it is \emph{outside} $Q_{n}$. 
\end{defn}

\begin{proof}[Proof of Theorem \ref{thm:main}]

By Proposition \ref{prop:circuit}, it suffices to prove $\{\omega : \tilde{\omega}\in A''_{n}\}\subset \mathcal{E}_{2n+10}$ for $n>100$. Thus we fix an $\omega$ and suppose $\tilde{\omega}\in A''_{n}$. Our goal is to prove $\omega\in \mathcal{E}_{2n+10}$. Let configuration $\omega_{0}$ be obtained by letting $\omega_{0}=\omega$ inside $Q_{100}$ and $\omega_{0}=\tilde{\omega}$ outside $Q_{100}$. Recall that $L(\omega_{0})$ is the trajectory of the light under the configuration $\omega_{0}$. By definition of $A''_{n}$ and $n>100$, the closed circuit of $\omega_{0}$ in $Q_{2n}\setminus Q_{n-1}$ gives rise to an enclosure of mirrors surrounding the origin and thus traps the light. Hence, $L(\omega_{0})$ is $2n$-bounded.

Note that, for each open edge $e$ in $\omega$,  $e$ is closed in the configuration $\omega_{0}$ if and only if
\begin{enumerate}
    \item $e$ is outside $Q_{100}$,
    \item $e$ is at the position of the (translated) red edge in a translated  copy of $G_{0}$.
\end{enumerate}

We denote by $E\subset \widetilde{\E^{2}}$ the set of edges  which are open in $\omega$ but are closed in $\omega_{0}$. 
\begin{figure}
    \centering
    \includegraphics{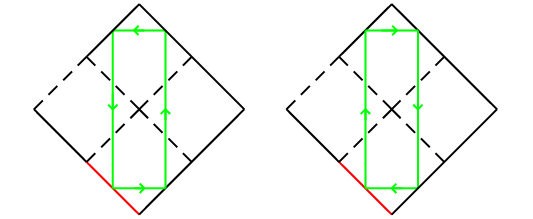}
    \caption{
    $e_{*}\in E$ is a (translated) red edge which is open under $\omega$. After enhancement, $e_{*}$ is turned to be closed. After closing $e_{*}$, the light comes from the north of $e_{*}$ will travel along the green loop (see the first figure); and the light comes from the east of $e_{*}$ will travel along the \emph{inverse} of green loop (see the second figure).}
    \label{fig:green-loop}
\end{figure}
\begin{cla}\label{cla:direction-of-light}
Suppose $e_{*}\in E$ and the light following the trajectory $L(\omega_{0})$ approaches $e_{*}$. Then the light must approach $e_{*}$ from the west or south.
\end{cla}
\begin{proof}
Let $\overline{G_{0}}$ be the translated copy of $G_{0}$ which appears in $\omega$ such that $e_{*}$ is at the position of the red edge in $\overline{G_{0}}$. By definition of $E$, $e_{*}$ is open in $\omega$ but closed in $\omega_{0}$. By Lemma \ref{lem:translation}, among all edges in $\overline{G_{0}}$, only $e_{*}$ belongs to $E$. Thus, except $e_{*}$, other edges in $\overline{G_{0}}$ have the same open-closed status in $\omega$ and $\omega_{0}$. Assume our claim is not true, then the light comes from the east or north of $e_{*}$. Then $L(\omega_{0})$ must be the green loop or its inverse as in Figure \ref{fig:green-loop} and its caption. Observe that, for any two points $a_{1},a_{2}\in \R^{2}$ on the green loop (or its inverse), we have $|a_{1}-a_{2}|\leq \sqrt{10}$. Since $L(\omega_{0})$ starts from the origin, any point $(x,y)$ on $L(\omega_{0})$ satisfies $|x+y-1|\leq 2\sqrt{10}+1<10$ and $|x-y|\leq 2\sqrt{10}<10$. We deduce from Definition \ref{def:n_bounded} that $L(\omega_{0})$ is $10$-bounded. Since $e_{*}$ intersects $L(\omega_{0})$, $e_{*}$ is inside $Q_{11}$. This contradicts with the fact that edges in $E$ are outside $Q_{100}$.
\end{proof}
We now start from $\omega_{0}$ and open edges in $E$ one by one to reach the initial configuration $\omega$ and we keep track of the trajectory. Suppose $e_{1},e_{2},\cdots,e_{k}$ are all the edges in $E$ which intersect $L(\omega_{0})$. For $1\leq i\leq k$, let $\omega_{i}$ be obtained from $\omega_{0}$ by opening edges $e_{1},e_{2},\cdots,e_{i}$. 
\begin{figure}
    \centering
    \includegraphics{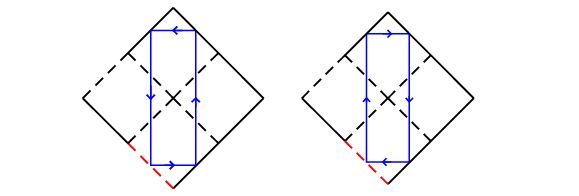}
    \caption{After opening the red edge of $G_{0}$, the light comes from the west of the red edge will travel an extra blue path (see the first figure); and the light comes from the south of the red edge will travel an extra path which is the inverse of blue path (see the second figure). After travelling the extra path, the light will continue along its original trajectory (i.e. the trajectory when the red edge was still closed).}
    \label{fig:blue-paths}
\end{figure}
\begin{cla}\label{cla:attach-blue-path}
Suppose $0\leq i\leq k-1$. Then
$L(\omega_{i+1})$ is obtained from $L(\omega_{i})$ by attaching a translated copy of the blue path (the first figure in Figure \ref{fig:blue-paths}) or its inverse (the second figure in Figure \ref{fig:blue-paths}). Moreover, this newly attached path intersects $e_{i+1}$ but does not intersect other edges in $E$.
\end{cla}
\begin{proof}
We prove by induction on $i$. For $i=0$, let $\overline{G_{0}}$ be the translated copy of $G_{0}$ which appears in $\omega$ and has translated red edge $e_{1}$. By Lemma \ref{lem:translation}, among edges of $\overline{G_{0}}$, only $e_{1}$ belongs to $E$ and other edges do not belong to $E$. Thus
the enhancement will only turn $e_{1}$ to be closed and will not affect other edges in $\overline{G_{0}}$. Suppose a light is traveling along $L(\omega_{0})$. By Claim \ref{cla:direction-of-light}, the light will approach $e_{1}$ from the west or south. Now, under configuration $\omega_{0}$, we open the edge $e_{1}$ and consider how the trajectory of light is affected. If the light comes from the west of $e_{1}$, it will travel an extra path which is a translated copy of the blue path (the first figure in Figure \ref{fig:blue-paths}); if the light comes from the south of $e_{1}$, it will travel an extra translated copy of the inverse of blue path (the second figure in Figure \ref{fig:blue-paths}). Thus $L(\omega_{1})$ is obtained from $L(\omega_{0})$ by attaching a translated copy of the blue path (or its inverse) which intersects $e_{1}$ but does not intersect other edges in $E$. 

Assume our claim holds for $i<k-1$ and we prove it for $i+1$. By inductive hypothesis, $L(\omega_{i+1})$ is obtained from $L(\omega_{0})$ by attaching $i+1$ translated copies of blue path (or its inverse). Thus the light traveling along $L(\omega_{i+1})$ will still approach $e_{i+2}$. By the same argument as in the proof of Claim \ref{cla:direction-of-light}, the light must approach $e_{i+2}$ from the west or south. By the same argument as in the case $i=0$, after opening $e_{i+2}$, the light will travel an extra path which is a translated copy of the blue path (or its inverse). Thus $L(\omega_{i+2})$ is obtained from $L(\omega_{i+1})$ by attaching a translated copy of the blue path (or its inverse) which intersects $e_{i+2}$ but does not intersect other edges in $E$. By induction principle, our claim follows.
\end{proof}
\begin{cla}\label{cla:one-by-one}
For each $1\leq i\leq k$
and each $e\in E$, if $e$ intersects $L(\omega_{i})$, then $e\in\{e_{1},e_{2},\cdots,e_{k}\}$. In particular, we have $L(\omega_{k})=L(\omega)$.
\end{cla}
\begin{proof}
By Claim \ref{cla:attach-blue-path}, $L(\omega_{i})$ is obtained from $L(\omega_{0})$ by attaching $i$ translated copies of the blue path or its inverse. If $e\in E$ intersects $L(\omega_{i})$, then either $e$ intersects $L(\omega_{0})$ or $e$ intersects those $i$ newly attached translated copies of blue path (or its inverse).
If $e\in E$ intersects $L(\omega_{0})$, then $e\in \{e_{1},e_{2},\cdots,e_{k}\}$ by definition of $E$. Otherwise, suppose $e$ intersects a newly attached translated copy of the blue path (or its inverse). By Claim \ref{cla:attach-blue-path}, this newly attached path intersects only one edge in $E$ which is one of $e_{i}$'s. Thus we have $e\in \{e_{1},e_{2},\cdots,e_{k}\}$.


Taking $i=k$, we deduce that for any edge $e$ in $E$, if $e$ intersects $L(\omega_{k})$, then $e\in \{e_{1},e_{2},\cdots,e_{k}\}$. Thus edges in $E\setminus \{e_{1},e_{2},\cdots,e_{k}\}$ do not intersect $L(\omega_{k})$ and opening these edges will not affect the original trajectory of light $L(\omega_{k})$. 
Since $\omega$ is obtained from $\omega_{k}$ by 
opening edges in $E\setminus \{e_{1},e_{2},\cdots,e_{k}\}$, we have $L(\omega_{k})=L(\omega)$. 
\end{proof}
\begin{figure}
    \centering
    \includegraphics{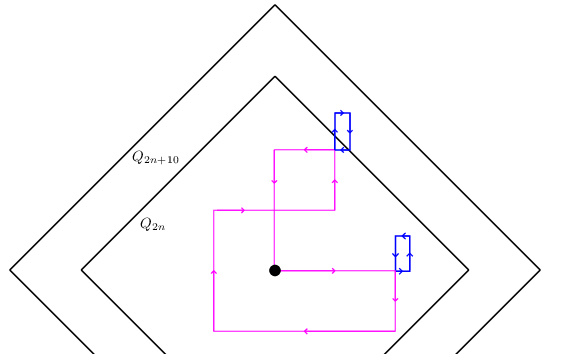}
    \caption{$L(\omega)$ is obtained from $L(\omega_{0})$ (the pink path) by attaching translated copies of the blue path (the first figure in Figure \ref{fig:blue-paths}) or its inverse (the second figure in Figure \ref{fig:blue-paths}) to $L(\omega_{0})$.}
    \label{fig:6}
\end{figure}
Finally, by Claim \ref{cla:one-by-one} and Claim \ref{cla:attach-blue-path}, $L(\omega)=L(\omega_{k})$ is obtained from $L(\omega_{0})$ by attaching $k$ translated copies of the blue path (or its inverse) to $L(\omega_{0})$ (see Figure \ref{fig:6}). Here, we used the fact that each newly attached path intersects one of $e_{i}$'s and $L(\omega_{0})$ intersects each of $e_{i}$'s.
By Figure \ref{fig:blue-paths}, observe that, for any two points $a_{1},a_{2}\in \R^{2}$ on the blue path (or its inverse), we have $|a_{1}-a_{2}|\leq \sqrt{10}$. Thus for any point $(x,y)\in \R^{2}$ on $L(\omega)$, there is a point $(x',y')\in \R^{2}$ on $L(\omega_{0})$ such that $|(x,y)-(x',y')|\leq \sqrt{10}$. Since $L(\omega_{0})$ is $2n$-bounded, we have
\begin{equation}\label{eq:check_bdd_1}
    |x+y-1|\leq |x'+y'-1|+|x-x'|+|y-y'|\leq |x'+y'-1|+2\sqrt{10} \leq 2n+2\sqrt{10}
\end{equation}
and 
\begin{equation}\label{eq:check_bdd_2}
    |x-y|\leq |x'-y'|+|x-x'|+|y-y'|\leq |x'-y'|+2\sqrt{10}\leq 2n+2\sqrt{10}.
\end{equation}
Here, the last inequalities in \eqref{eq:check_bdd_1} and \eqref{eq:check_bdd_2} are by the definition of $2n$-bounded (see Definition \ref{def:n_bounded}). 
Since $2n+2\sqrt{10} < 2n+10$, we have $L(\omega)$ is $(2n+10)$-bounded and our theorem follows.
\end{proof}
\begin{appendix}
\section{Essential enhancement}\label{app:enh}
We give here a sketch of the proof of Proposition \ref{prop:event-A_n} and the same argument works for the proof of Proposition \ref{prop:event-A'_n} (see Remark \ref{rem:A'_n_case}). We follow the argument in \cite[Section 3.3]{grimmett1999percolation} and adapt notations in \cite[Section 3.3]{grimmett1999percolation} while keeping in mind that our enhancement is to add more closed edges rather than open edges. 

Denote by $\Omega$ the edge configuration space. Each configuration $\omega\in \Omega$ can be viewed as a function $\omega:\widetilde{\E^{2}} \rightarrow \{0,1\}$ such that $\omega(e)=1$ if $e$ is closed under $\omega$ and $\omega(e)=0$ otherwise. For any $e\in \widetilde{\E^{2}}$ and $x\in \widetilde{\Z^2}$, we denote by $e+x$ the translate of the edge $e$ by the vector $x\in \widetilde{\Z^{2}}$; similarly, for any subgraph $G$ of $\widetilde{\Z^{2}}$, let $G+x$ denote the translate by $x$ of $G$.

Let $\mathcal{G}$ be the set of all simple graphs on the vertex set $Q:=Q_{1}$. For any configuration $\omega$, let $\omega_{Q} \in \mathcal{G}$ be the graph on $Q$ with edge set $$\{e \in \widetilde{\E^2}:e \text{ is inside $Q$ and } \omega(e)=1\}.$$ Let $\mu_{tri} \in \mathcal{G}$ be the graph on $Q$ with empty edge set. 
Let $\mu_{*},\mu_{**} \in \mathcal{G}$ be the graphs defined in Figure \ref{fig:enh-function} and its caption.
Let $F:\mathcal{G}\rightarrow \mathcal{G}$ be the \emph{enhancement function} defined as follows:
\begin{equation}
    F(\mu)=
    \begin{cases}
    \mu_{**} &\text{if } \mu=\mu_{*}\\
    \mu_{tri} &\text{otherwise}.
    \end{cases}
\end{equation}
For each $x\in \widetilde{\Z^{2}}$ and $\omega\in \Omega$, we write $F(x,\omega)=F((\tau_{x}\omega)_{Q})$ where $\tau_{x}$ is the shift operator on $\Omega$ given by $\tau_{x}\omega(e)=\omega(e+x)$. 
Let $\Xi=\{0,1\}^{\widetilde{\Z^{2}}}$. For each $\omega\in \Omega$ and $\xi \in \Xi$, denote the enhanced configuration by (more precisely, the graph of closed edges under the enhanced configuration)
\begin{equation}\label{eq:define-enh}
    G^{\textbf{enh}}(\omega,\xi)=G(\omega) \cup \left\{\bigcup_{x:\xi(x)=1} \{x+F(x,\omega)\} \right\}
\end{equation}
where $G(\omega)$ is the graph of closed edges under $\omega$. In writing the union of graphs, we mean the graph with vertex set $\widetilde{\Z^{2}}$ having the union of the appropriate edge sets; wherever this union contains two or more edges between the same pair of vertices, these edges coalesce into a single edge.

Thus we have associated with each pair $(\omega,\xi)\in \Omega \times \Xi$ an enhanced graph $G^{\textbf{enh}}(\omega,\xi)$.
We consider making an enhancement stochastically at each vertex $x\in \widetilde{\Z^{2}}$. 
Suppose $p,s\in [0,1]$. For each $e\in \widetilde{\E^{2}}$ and each $x\in \widetilde{\Z^{2}}$, we declare $\omega(e)=1$ with probability $p$ and declare $\xi(x)=1$ with probability $s$, independent of all other edges and vertices. This endows the space $\Omega \times \Xi$ with the product probability measure $\Prob_{p,s}$. 
\begin{figure}
    \centering
    \includegraphics{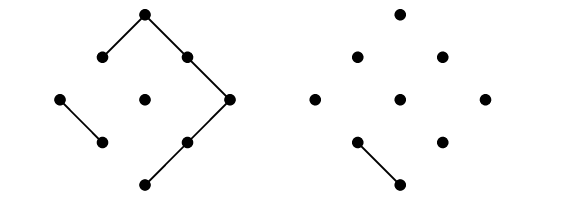}
    \caption{The left hand side illustrates graph $\mu_{*}$ and the right hand side illustrates graph $\mu_{**}$. In both graphs, the black dots illustrate vertices in $Q=Q_{1}$ and the segments illustrate the edges.}
    \label{fig:enh-function}
\end{figure}
\begin{proof}[Proof of Proposition \ref{prop:event-A_n}]
Given two configurations $\omega,\omega' \in \Omega$, we write $\omega \leq \omega'$ if $\omega\leq \omega'$ as functions in $\{0,1\}^{\widetilde{\E^2}}$.
Since $A_{n}\subset \{\omega : \tilde{\omega}\in A_{n}\}$, we have $\Prob_{p}[\{\omega : \tilde{\omega}\in A_{n}\}] \geq \Prob_{p}[A_{n}]$ for each $p\in (0,1)$. Thus for $p\geq \frac{1}{2}+\varepsilon_{1}$, we have 
\begin{equation}
    \Prob_{p}[\{\omega : \tilde{\omega}\in A_{n}\}] \geq \Prob_{p}[A_{n}] \geq \Prob_{\frac{1}{2}+\varepsilon_{1}}[A_{n}]
\end{equation}
where the second inequality is because $A_{n}$ is an increasing event. i.e. For any two configurations $\omega\leq \omega'$ such that $\omega\in A_{n}$, we have $\omega'\in A_{n}$. Thus it suffices to prove our proposition under the condition $p\in (\frac{1}{2}-\varepsilon_{1},\frac{1}{2}+\varepsilon_{1})$.

Let $B_{n}\subset \Omega\times \Xi$ denote the event that $(\frac{1}{2},\frac{1}{2})$ is connected to $\partial Q_{n}$ in the graph $G^{\textbf{enh}}(\omega,\xi)$. Define $\xi_{0},\xi_{*}\in \Xi$ by letting $\xi_{0}(x)=0$ and $\xi_{*}(x)=1$ for each $x\in \widetilde{\Z^{2}}$. Then the edges of $G^{\textbf{enh}}(\omega,\xi_{0})$ are exactly the closed edges of $\omega$ and the edges of $G^{\textbf{enh}}(\omega,\xi_{*})$ are closed edges of $\tilde{\omega}$.
Thus $\omega\in A_{n}$ is equivalent to $(\omega,\xi_{0})\in B_{n}$ and $\tilde{\omega}\in A_{n}$ is equivalent to $(\omega,\xi_{*})\in B_{n}$. This implies 
\begin{equation}
    \Prob_{p}[A_{n}]=\Prob_{p,0}[B_{n}]
\end{equation}
and
\begin{equation}
    \Prob_{p}[\{\omega:\tilde{\omega}\in A_{n}\}]=\Prob_{p,1}[B_{n}].
\end{equation}
Hence, it suffices to prove that, there is $\varepsilon_{1}>0$ such that
\begin{equation}\label{eq:goal}
    \Prob_{p,1}[B_{n}] \geq \Prob_{\frac{1}{2}+\varepsilon_{1},0}[B_{n}]
\end{equation}
for each $p\in (\frac{1}{2}-\varepsilon_{1},\frac{1}{2}+\varepsilon_{1})$ and large enough $n$.
Denote $\theta_{n}(p,s)=\Prob_{p,s}(B_{n})$ for each $p,s\in [0,1]$. 
By adapting the same argument in the proof of \cite[Theorem (3.16)]{grimmett1999percolation} (from Page 67 to Page 71 in \cite{grimmett1999percolation}), we deduce \cite[equation (3.25)]{grimmett1999percolation}. Namely, there exists a continuous function $\nu : (0,1)\times (0,1) \rightarrow (0,\infty)$ such that 
\begin{equation}\label{eq:critical-ine}
    \frac{\partial \theta_{n}}{\partial p} (p,s) \leq \nu(p,s) \frac{\partial \theta_{n}}{\partial s} (p,s)
\end{equation}
for each $(p,s)\in (0,1)\times (0,1)$ and each $n>n_{0}$ where $n_{0}$ is a constant independent of $p,s$. Now we pick $\eta \in (0,\frac{1}{10})$ and choose $K>1$ such that $\nu(p,s)\leq K$ on $[\eta,1-\eta]^{2}$. Let $\varepsilon_{1}>0$ be small enough such that $\frac{1}{2}+\varepsilon_{1} K <1-\eta$. Denote function $\phi(p)=\frac{1}{2}-K (p-\frac{1}{2})$. Suppose $n>n_{0}$. Then for $p\in [\frac{1}{2}-\varepsilon_{1},\frac{1}{2}+\varepsilon_{1}]$, we have $(p,\phi(p))\in [\eta,1-\eta]^{2}$ and $\nu(p,\phi(p))\leq K$. Since $\theta_{n}(p,s)$ is a polynomial of $p,s$, we have for $p\in (\frac{1}{2}-\varepsilon_{1},\frac{1}{2}+\varepsilon_{1})$,
\begin{align}
\begin{split}\label{eq:decreasing-theta}
    \frac{d}{d p} \theta_{n}(p,\phi(p))=&\frac{\partial \theta_{n}}{\partial p} (p,\phi(p)) -K \frac{\partial \theta_{n}}{\partial s} (p,\phi(p))\\
        &\leq (\nu(p,\phi(p))-K) \frac{\partial \theta_{n}}{\partial s} (p,\phi(p))\\
        &\leq 0.
\end{split}
\end{align}
Here, we also used the fact that $\frac{\partial \theta_{n}}{\partial s}(p,s)\geq 0$ for $p,s\in (0,1)$ (which is a direct consequence of \cite[equation (3.17)]{grimmett1999percolation}). By \eqref{eq:decreasing-theta}, $\theta_{n}(p,\phi(p))$ is non-increasing in $p$ for $p\in [\frac{1}{2}-\varepsilon_{1},\frac{1}{2}+\varepsilon_{1}]$. Thus we have 
\begin{equation}
    \theta_{n}\left(\frac{1}{2}+\varepsilon_{1},0\right)\leq \theta_{n}\left(\frac{1}{2}+\varepsilon_{1},\phi\left(\frac{1}{2}+\varepsilon_{1}\right)\right)\leq  \theta_{n}(p,\phi(p)) \leq \theta_{n}(p,1) 
\end{equation}
for $p\in (\frac{1}{2}-\varepsilon_{1},\frac{1}{2}+\varepsilon_{1})$. Here, the first and third inequality is because $\theta_{n}(p,s)$ is non-decreasing in $s$ (since $\frac{\partial \theta_{n}}{\partial s}(p,s)\geq 0$ for $p,s\in (0,1)$). Hence, \eqref{eq:goal} holds and our proposition follows.
\end{proof}
\begin{rem}\label{rem:A'_n_case}
In order to adapt the above argument to  prove Proposition \ref{prop:event-A'_n}, substitute $A_{n}$ by $A'_{n}$ and substitute $B_{n}$ by $B'_{n}$ where $B'_{n}$ denotes the following event: under the graph $G^{\textbf{enh}}(\omega,\xi)$, there is a path inside $T_{n}$ that joins the northwest side of $T_{n}$ to the southeast side of $T_{n}$. In order to achieve \eqref{eq:critical-ine}, adapt the argument from Page 67 to Page 71 in \cite{grimmett1999percolation} to the event $A=B'_{n}$.
\end{rem}
\end{appendix}
\section*{Acknowledgement}
The author thanks Lingfu Zhang for telling the problem to him and having several discussions. The author also thanks Professor Jian Ding for useful comments on writing and thanks Professor Geoffrey Grimmett for reading the early draft and giving helpful suggestions. The author is partially supported by NSF grant  DMS-1757479.
\bibliographystyle{style}
\bibliography{bib}

\end{document}